\documentclass[11pt]{article}
\usepackage{amsmath}
\usepackage{amsthm,amssymb,url,hyperref}
\usepackage{fullpage}
\usepackage[all]{xy}
\usepackage{xcolor}

\usepackage{tabularx}
\theoremstyle{plain}
\newtheorem{theorem}{Theorem}[section]

\newtheorem{corollary}[theorem]{Corollary}

\newtheorem{lemma}[theorem]{Lemma}

\newtheorem{proposition}[theorem]{Proposition}

\theoremstyle{definition}
\newtheorem{definition}[theorem]{Definition}

\newtheorem{example}[theorem]{Example}

\def\Aut#1{\mathrm{Aut}(#1)}

\DeclareMathOperator{\Img}{Im}
\DeclareMathOperator{\Inn}{Inn}
\DeclareMathOperator{\Hom}{Hom}
\DeclareMathOperator{\Triv}{Triv}

\def\hom_P#1{Hom_\mathcal{P}(#1)}
\def\setof#1#2{\{#1\, : \,#2\}}

\def\cg#1{\equiv_\alpha}
\newcommand*\xbar[1]{%
   \hbox{%
     \vbox{%
       \hrule height 0.5pt 
       \kern0.5ex
       \hbox{%
         \kern-0.1em
         \ensuremath{#1}%
         \kern-0.1em
       }%
     }%
   }%
} 

\title{On the Structure of Hom Quandles}

\author{Marco Bonatto \\ Charles University Prague \footnote{\texttt{marco.bonatto.87@gmail.com}}
\and Alissa S. Crans \\ Loyola Marymount University \footnote{\texttt{acrans@lmu.edu} Alissa S. Crans was supported by a grant from the Simons Foundation (\#360097, Alissa Crans). } 
\and Glen Whitney \\ Harvard University \footnote{\texttt{gwhitney@math.harvard.edu}}}

\begin{document}
\maketitle

\begin{abstract}
We continue the study of the quandle of homomorphisms into a medial quandle begun in \cite{CN}.  We show that it suffices to consider only medial source quandles, and therefore the structure theorem of \cite{JPSZ} provides a characterization of the Hom quandle. In the particular case when the target is 2-reductive this characterization takes on a simple form that makes it easy to count and determine the structure of the Hom quandle. 
\end{abstract}

\section{Introduction}
It is natural to study the space of morphisms between two algebraic structures of the same kind. This study is particularly rewarding when these ``Hom-sets" themselves support the same algebraic structure as the objects they relate. In particular, this framework occurs in the study of quandles. In classical group theory, we know that the homomorphisms from $G$ to $H$ form a group under pointwise operations precisely when $H$ is abelian. In \cite{CN}, Crans and Nelson demonstrate the analogous phenomenon in the context of quandles, that is, the target quandle must satisfy a certain sort of commutativity property (known as {\em mediality}) in order for the Hom-set to be a quandle.  This concept of mediality (sometimes suitably generalized and/or under 
other names, such as {\it entropicity}) arises for various algebraic 
structures, in large part because of its connection with imposing a 
similar algebraic structure on Hom-sets under pointwise operations.  See for example \cite{RS} or the work on quasigroups in \cite{Mur}.

Inspired by this fact, we sought to understand in more depth how the properties of quandles $S$ and $T$ influence those of $\Hom(S,T)$, when the latter is a quandle. Moreover, numerous papers in the literature study knot and link colorings, which are none other than homomorphisms from certain quandles fundamentally associated to these objects.  Our focus here is primarily on the algebraic aspects. Crans and Nelson demonstrated, for example, that $\Hom(S,T)$ inherits the properties of being commutative and involutory from the quandle $T$. One initial point of curiosity concerned the orders of $S, T,$ and $\Hom(S,T).$ To that end, we generated a table (using the RIG package in GAP \cite{V}) of the cardinalities of all Hom quandles for small $S$ and $T.$  A portion of this table is reproduced as Table \ref{table}.  The non-trivial quandles of size $n$ are labeled as $Qn_i$ where $i$ is an index number.  For quandles of size less than six, we use the ordering from \cite{HN} reading their tables in Figures 3 and 4 in row-major order.  For the quandles of order six, we use the index from the list in Table 2 of \cite{Elhamdadi}.
 \begin{table}[hb!]  \setlength{\tabcolsep}{4pt}
  \begin{tabular}{>{$}c<{$} | c c c c c c c  c c c c  c c c c c c  } \label{table} 

  S \setminus T&$I$&$2I$&$3I$&$Q3_2$&$Q3_3$&4I&$Q4_2$&$Q4_3$&$Q4_4$&$Q4_6$&$Q4_7$&$Q5_3$&$Q5_7$&$Q5_8$&$Q5_{16}$& $Q6_{67}$ & $Q6_{71}$\\ \hline
 
I&1&2&3&3&3&4&4&4&4&4&4&5&5&5&5&6 & 6\\ 
2I&1&4&9&3&5&16&12&10&8&8&4&19&13&13&13&18 & 12 \\
3I&1&8&27&3&9&64&34&28&16&16&4&71&35&35&35&54& 24 \\
Q3_2&1&2&3&9&3&4&4&4&4&4&4&5&5&5&5&6 & 6 \\
Q3_3&1&4&9&3&7&16&14&10&12&8&4&19&13&13&13&18 & 12\\
4I&1&16&81&3&17&256&96&82&32&32&4&271&97&97&97&162 & 48 \\
Q4_2&1&8&27&3&9&64&36&28&16&16&4&71&35&35&35&54  & 24\\
Q4_3&1&4&9&3&5&16&12&13&8&8&4&22&19&19&13&18 & 12\\
Q4_4&1&8&27&3&11&64&36&28&24&16&4&71&35&35&35&54 & 24 \\
Q4_5&1&4&9&3&5&16&12&10&8&8&4&19&13&13&13&18 & 12\\
Q4_6&1&4&9&3&9&16&16&10&16&16&4&19&13&13&13&18 & 20\\
Q4_7&1&2&3&3&3&4&4&4&4&4&16&5&5&5&5&6 & 6\\
Q5_3&1&8&27&3&9&64&34&28&16&16&4&74&35&35&35&54 & 24\\
Q5_7&1&8&27&3&9&64&34&31&16&16&4&74&47&41&35&54 & 24\\
Q5_8&1&8&27&3&9&64&34&28&16&16&4&71&35&41&35&54 & 24\\
Q5_{16}&1&4&9&3&7&16&14&13&12&8&4&22&19&19&19&18 & 12\\
Q6_{52} &1 & 4& 9& 3& 5& 16& 12& 10& 8& 8& 4& 19& 13 & 13 & 13 & 18 &  12  \\
Q6_{67}&1&4&9&3&5&16&12&16&8&8&4&25&25&25&13&36  & 12\\
Q6_{71} & 1 & 4 & 9 & 3 & 9 & 16 & 16 & 10 & 16 & 16 & 4 & 19 & 13 & 13 & 13 & 18 & 28 \\
\end{tabular} 
\caption{Cardinalities of $\Hom(S,T)$}

\end{table}

We immediately notice that some columns (corresponding to specific target quandles) exhibit more variation than others. It's perhaps not surprising that the entries with trivial targets ($I$, $2I$, $3I,$ $\ldots,$ see Section \ref{prelim} for precise definitions) are all powers of the size of the target, but which powers are they? Certain numbers seem to appear much more frequently in the table than others. We explain these observations via the results in this paper.

We begin in Section \ref{prelim} with a brief review of basic quandle definitions and facts, primarily to establish the notational conventions of this paper. We continue in Section \ref{SecQuandleHom} with a study of the structure of Hom quandles.  We begin by explaining the relationship between components of the source and components of the target.  Next, in Section \ref{homquandle} we add to the collection of properties in \cite{CN} that $\Hom(S,T)$ inherits from $T$ and count the number of homomorphisms with trivial image.  In Section \ref{mainsection} we have our main result, Theorem \ref{maintheorem}, relating the homomorphism quandle to group homomorphisms between the components of the source and target.  We continue in Section \ref{homsource} by showing that for a wide class of identities, every homomorphism into a quandle satisfying those identities factors through a quotient of the source, where the quotient also satisfies those identities.   We put these ingredients together in Section \ref{main} in Corollary \ref{criterion2}, which precisely counts and characterizes $\Hom(S,T)$ for $T$ a 2-reductive quandle and arbitrary $S$.  Concrete examples illustrating these results within Table \ref{table} appear throughout.

\section{Preliminaries}\label{prelim}
We begin by reviewing definitions and well-known facts about quandles.  

A {\it quandle} is a set $Q$ equipped with a binary operation $\rhd$ that satisfies the following three axioms:
\begin{itemize}
\item $x \rhd x =x$ for all $x \in Q$ (idempotence),
\item for each $y, z \in Q$, there exists a unique $x \in Q$ such that $y \rhd x = z$ (left divisibility), and 
\item $x \rhd (y \rhd z) = (x \rhd y) \rhd (x \rhd z)$ for all $x,y,z \in Q$ (self-distributivity)
\end{itemize}
Each element $a\in Q$ defines a map $L_a:Q \to Q$ by $L_a(x)= a \rhd x$. Then, the left divisibility axiom implies that each $L_a$ is a bijection and the self-distributivity implies that each $L_a$ is a quandle homomorphism, and therefore an automorphism. The {\it inner} automorphism group of $Q$, denoted by $\Inn(Q)$, is the normal subgroup of $\text{Aut}(Q)$ generated by the inner automorphisms $L_a$. 

A quandle $Q$ is {\it trivial} if $x \rhd y = y$ for all $x,y \in Q$. Up to isomorphism, there is a unique trivial quandle of a given order.  We denote the one element trivial element by $I$ and the $n$-element one by $nI$. Equivalently, a quandle is trivial if every component is a singleton.  Numerous examples of non-trivial quandles (Alexander, dihedral, Latin, etc.) can be found in the literature \cite{EGS, Jsimple,  G,Sta-latin}.

The orbits with respect to the action of $\Inn(Q)$ on $Q$ are called the {\it components}, or {\it orbits}, of $Q$.  We will refer to a set consisting of one representative from each component as a {\it set of base points of $Q$}. A quandle with only one component is called {\it connected}.  We denote by $c(Q)$ the set of the components, which when convenient we think of as endowed with the trivial quandle structure.  We note that the map $c_Q: Q \longrightarrow c(Q)$ defined by $a \mapsto a^{\Inn(Q)}$ is a homomorphism of quandles.  Accordingly, we will denote the component of $a \in Q$ by $c_Q(a)$ rather than $a^{\Inn(Q)}$ for brevity.

\section{Quandle Homomorphisms}\label{SecQuandleHom}
We commence our investigation of quandle homomorphisms with a relatively simple observation.

\begin{lemma}\label{component to component}
Let $S$ and $T$ be quandles and $h:S\to T$ be a quandle homomorphism.  If $a$ and $b$ are in the same component of $S$, then $h(a)$ and $h(b)$ are in the same component of $T$.
\end{lemma}
\begin{proof}
Since $b$ is in the same component as $a$, then $b = a_1 \rhd (a_2 \rhd  \ldots (a_n \rhd a))$ for some $a_i \in S$.  Then, 
$h(b)=h(a_{1}) \rhd (h(a_{2}) \rhd \ldots (h(a_{n}) \rhd h(a))),$
so $h(b)$ is in the same component as $h(a)$.
\end{proof} 
In fact, the above argument shows that $h(a)$ and $h(b)$ are in the same orbit with respect to the subgroup of Inn$(T)$ generated by the image of $h$.

Lemma \ref{component to component} already illuminates some of the patterns we saw in Table \ref{table}. Since the components of a trivial target are singletons, Lemma \ref{component to component} implies that the power of the size of the target in the entries for trivial targets $2I$, $3I$, etc. should be the number of components of the source, a fact which easily can be verified through inspection of those columns. The remainder of this section is devoted to making this observation precise and generalizing it.

If $f$ is a quandle homomorphism from $S$ to $T$, we can define an equivalence relation $\sim$ on $S$ by $a \sim b$ if and only if $f(a) = f(b)$ for all $a, b \in S$.  This equivalence relation is called the {\it kernel} of $f$.   We say $\sim$ is a {\it congruence} if whenever $a \sim b$ and $c \sim d$, then $a \rhd c \sim b \rhd d$.   The kernel of any homomorphism is a congruence.  Given a quandle $S$ and a congruence, we can form the quotient quandle $S/\! \!\sim$ whose elements are the equivalence classes with the induced operation and this is well-defined by the congruence property.  

In the case of the homomorphism $c_Q$ from Section \ref{prelim}, $\ker(c_Q)$ is the minimal congruence such that the quotient is trivial.  That is to say, for any congruence $\alpha$, $Q/\alpha$ is trivial if and only if $\ker(c_Q) \subseteq \alpha$.  On the other hand for congruences contained in $\ker(c_Q)$ we have:

\begin{proposition}\label{number_of_comp}
Let $Q$ be a quandle and $\alpha \subseteq \ker(c_Q)$ a congruence on $Q$. Then $c_Q(a) = c_Q(b)$ if and only if $c_{Q/\alpha}([a]_{\alpha}) = c_{Q/\alpha}([b]_{\alpha})$ and hence $|c(Q)|=|c(Q/\alpha)|$.
\end{proposition}
\begin{proof}
	We will show we have a bijection between $c(Q)$ and $c(Q/\alpha)$. The following diagram commutes:
\begin{displaymath}
    \xymatrixcolsep{63pt}\xymatrixrowsep{30pt}\xymatrix{
    Q\ar[r]^{\pi_\alpha}
    \ar[d]^{c_Q} & Q/\alpha\ar[d]^{c_{Q/\alpha}} \\
    c(Q) \ar[r]^{\phi}& c(Q/\alpha)}
\end{displaymath}
where $\phi(c_Q(a))=c_{Q/\alpha}([a]_\alpha)$. Moreover, since 
$\alpha \subseteq \mathrm{ker}(c_Q)$, the following diagram also commutes:
\begin{displaymath}
    \xymatrixcolsep{63pt}\xymatrixrowsep{30pt}\xymatrix{
    Q\ar[d]^{c_Q}
    \ar[r]^{\pi_\alpha} & Q/ \alpha \ar[d]^{c_{Q/ \alpha}} \\
    c(Q)  & c(Q/\alpha) \ar[l]^{\psi}}
\end{displaymath}
where $\psi(c_{Q/\alpha}([a]_\alpha))=c_{Q}(a)$. Thus, the maps $\phi$ and $\psi$ are inverses of one another.  \end{proof}

Now, suppose $T$ satisfies an identity $p(x_1, \ldots, x_n) = q(x_1, \ldots, x_n)$ and $h \in$ Hom$(S,T)$.  Then, for any $x_1, \ldots, x_n \in S$, 
 \begin{eqnarray*}
h(p(x_1,\ldots, x_n))&=& p(h(x_1),\ldots, h( x_n))\\
h(q(x_1,\ldots, x_n))&=&q(h(x_1),\ldots, h(x_n)).
\end{eqnarray*}
Therefore, $(p(x_1,\ldots, x_n),q(x_1,\ldots, x_n)) \in \textrm{ker}(h)$.  

For any collection of identities $K$, we will denote by Cg$(K)$ the congruence generated by $K$.  That is, Cg$(K)$ is the minimal congruence such that $a \sim b$ whenever there exist $x_1, \ldots, x_n$ such that $a = p(x_1, \ldots, x_n)$ and $b = q(x_1, \ldots, x_n)$ with $p = q$ being an identity in $K$.  Now, we have the following result. 
 
\begin{theorem}\label{factoring}
Let $S$ and $T$ be quandles and let $K$ be the set of identities satisfied by $T$. Then $\Hom(S,T)\cong \Hom(S/\mathrm{Cg}(K),T)$ as sets.
\end{theorem}

\begin{proof}
Let $\alpha=\mathrm{Cg}(K)$. The remark above shows that for every $h \in \mathrm{Hom}(S,T)$,   $\alpha\subseteq$ ker$(h)$.  Thus, by the First Homomorphism Theorem, there exists a unique $\widetilde{h}\in \mathrm{Hom}(S/\alpha,T)$ such that $h=\widetilde{h}\circ \pi_\alpha$.
Thus, the map:
$$\psi:\mathrm{Hom}(S,T)\to \mathrm{Hom}(S/\alpha,T),\quad h\mapsto \widetilde{h},$$
is a bijection with inverse given by $f\mapsto f\circ \pi_\alpha$ for every $f\in \mathrm{Hom}(S/\alpha,T)$.
\end{proof}

\begin{example}\label{ex1}
Let $T$ be trivial, and $S$ be a quandle and $h\in \textrm{Hom}(S,T)$. Then, since
$h(a \rhd b)=h(a)\rhd h(b)=h(b),$ for $a, b \in S$,
$h$ is constant on each of the components of $S$. So, we have that $h$ factors through $S/\ker(c_S) = c(S)$, that is, $h$ corresponds to a map from $c(S)$ to $T$. Hence, $\textrm{Hom}(S,T)\cong T^{c(S)}$.
\end{example}

\subsection{Hom Quandles} \label{homquandle}
Thus far we have only considered Hom$(S,T)$ as a set.  In fact, it has a richer structure when $T$ is {\it medial}, meaning that $T$ satisfies the identity $(x \rhd y) \rhd (z \rhd w) = (x \rhd z) \rhd (y \rhd w)$ for all $x, y , z,$ and $w$.  By Theorem 3 of \cite{CN}, Hom$(S,T)$ is a medial quandle under the pointwise operation $(h \rhd k)(a) = h(a) \rhd k(a)$ when $T$ is medial.  Furthermore, Theorems 7 and  8 of \cite{CN} tell us that $T$ embeds into Hom$(S,T)$ and Hom$(S,T)$ embeds into $T^r$ where $r$ is the minimal cardinality of a set of generators of $S$, respectively. 

We begin by adding to the collection of properties in \cite{CN} that $\Hom(S,T)$ inherits from $T$.  Our first goal is to show that connectivity is such property, and to do so we need the following notion.  A quandle is called {\it Latin} when for each $y \in T$, the map $x \mapsto x \rhd y$ is a permutation. (The operation tables for such quandles are Latin squares.)

\begin{lemma} \label{on latin}
If $T$ is a finite, Latin quandle, then all subquandles of powers of $T$ are Latin.
\end{lemma}

\begin{proof} Since $T$ is finite, each map $R_y: x\to x\rhd y$ has finite order; let  $n$ be the least common multiple of all such orders.  We must show that for any index set $I$, every subquandle $S$ of $T^I$ is Latin. For any element $a = {\setof{y_i}{i\in I}} \in S$ the map $R_a : \setof{x_i}{i\in I} \mapsto \setof{x_i\rhd y_i}{i\in I}$ has order $m \leq n$. We must show that for every $b \in S$, the equation $x\rhd a=b$ has a unique solution.  But, we have $x\rhd a=R_a(x)=b$ if and only if $x=R_a^{-1}(b)=R_a^{m-1}(b)=((b\rhd a)\rhd \ldots )\rhd a\in S$. 
\end{proof}

\begin{lemma}  
Let $T$ be a finite, medial connected quandle. Then $\mathrm{Hom}(S,T)$ is connected for every quandle $S$.
\end{lemma}
\begin{proof} 
Since $T$ is connected and medial, then it is Latin \cite[Proposition 1]{Lith}.  By our previous lemma, so are all of its powers and their subquandles. Since $\mathrm{Hom}(S,T)$ embeds into a power of $T$, then it is connected.
\end{proof}

\begin{theorem}
Let $S$ be a quandle and $T$ be a medial quandle. Then $T$ satisfies an identity if and only if $\mathrm{Hom}(S,T)$ does.
\end{theorem}
\begin{proof}
By Theorems 7 and 8 of \cite{CN}, $\mathrm{Hom}(S,T)$ is a subquandle of a power of $T$, and $T$ is a subquandle of $\mathrm{Hom}(S,T)$. Hence $T$ satisfies an identity if and only if $\mathrm{Hom}(S,T)$ does.
\end{proof}

\begin{lemma}
Let $S,R$ be quandles and $T$ be a medial quandle and $h: S \rightarrow R$ be a homomorphism. Then 
$\mathrm{Hom}(R,T) \to \mathrm{Hom}(S,T)$ given by $k \mapsto k \circ h$
is a quandle homomorphism.
\end{lemma}

\begin{proof}
Clearly if $k \in \mathrm{Hom}(R,T)$, then $k \circ h \in \mathrm{Hom}(S,T)$. Moreover:
\begin{displaymath}
[(k\rhd l)\circ h] (a)=(k(h(a))\rhd l(h(a))=(k\circ h)(a)\rhd (l\circ h)(a)=[(k\circ h)\rhd (l\circ h)] (a)
\end{displaymath}
for every $a\in S$.  
\end{proof}

\begin{lemma}
Let $S$ be a quandle, $T$ and $U$ be medial quandles and $h \in \mathrm{Hom}(T,U)$. Then 
$\mathrm{Hom}(S,T) \to \mathrm{Hom}(S,U)$ given by $k \mapsto h \circ k$
is a quandle homomorphism.
\end{lemma}

\begin{proof}
Clearly if $k \in \mathrm{Hom}(S,T)$, then $h \circ k \in \mathrm{Hom}(S,U)$. Moreover:
\begin{displaymath}
h \circ (k\rhd l)(a)=h(k(a)\rhd l(a))=(h\circ k)(a)\rhd (h\circ l)(a)=[(h\circ k)\rhd (h\circ l)] (a)
\end{displaymath}
for every $a\in S$.  
\end{proof}

The previous two Lemmas could be restated as saying that $\Hom(-,T)$ is a functor from the category of quandles to the category of medial quandles and $\Hom(S, -)$ is an endofunctor of the category of medial quandles.  In particular, if $U$ is a subquandle of $T$, then $\mathrm{Hom}(S,U)$ is a subquandle of $\mathrm{Hom}(S,T)$. 

\begin{lemma}\label{trivial targets}
Let $S$ be a quandle and $T$ be a medial quandle. Then $$\mathrm{Triv}(S,T)=\setof{ h \in \mathrm{Hom}(S,T)} {\mathrm{Im}(h) \, \text{is trivial}}$$ is a subquandle of $\mathrm{Hom}(S,T)$.
\end{lemma}
\begin{proof}
Let $h, k \in \mathrm{Triv}(S, T)$.  Then,
\begin{eqnarray*}
(h\rhd k)(a)\rhd (h\rhd k)(b)&=&(h(a)\rhd k(a))\rhd (h(b)\rhd k(b)) =(h(a)\rhd h(b))\rhd (k(a)\rhd k(b)) \\
&=&h(b)\rhd  k(b) =(h\rhd k)(b)
\end{eqnarray*}
for every $a,b\in Q$. So $\mathrm{Im}(h \rhd k)$ is a trivial subquandle of $T$ meaning $\Triv(S,T)$ is closed under $\rhd$. It remains to show that the elements required by the second quandle axiom lie in $\Triv(S,T)$.  Let $l: S \rightarrow T$ be the unique homomorphism such that $h \rhd l = k$.  We need to show that for every $a, b \in S$, $l(a) \rhd l(b) = l(b)$.  Since $l$ is a homomorphism, $l(a) \rhd l(b) = l(a \rhd b)$, which is in turn the unique element in $T$ such that $h(a \rhd b) \rhd l(a \rhd b) = k(a \rhd b)$. On the other hand, $h(a \rhd b) \rhd l(b) = (h(a) \rhd h(b)) \rhd l(b) = h(b) \rhd l(b) = k(b) = k(a) \rhd k(b) = k(a \rhd b)$.  Thus, $l(b)$ is also the unique element with this property, i.e., $l(a) \rhd l(b) = l(b)$. So $\mathrm{Triv}(S,T)$ is a subquandle of $\mathrm{Hom}(S,T)$.
\end{proof}
Not only is $\mathrm{Triv}(S,T)$ a subquandle of the Hom quandle, but we can give a precise characterization of its elements.

\begin{lemma}\label{hom with projec image}
Let $S$ be a quandle and $T$ be a medial quandle. Then:
$$\Triv(S,T) \cong \setof{f:c(S)\to U}{U \text{ is a trivial subquandle of  } T \text{ and } f \text{ is surjective}}$$ as sets. If $U$ is a trivial subquandle of size $n \leq |c(S)|$, then there exist
\begin{equation}\label{stirling}
\frac{1}{n!}\sum_{j=1}^{n} (-1)^{n-j} {{n}\choose{j}} j^{|c(S)|}
\end{equation}
 homomorphisms with image equal to $U$.  
\end{lemma}

\begin{proof}
Let $h \in \Triv(S,T)$. Then $h$ factors through $S/\mathrm{ker}(c_S)$ as in Example \ref{ex1}.  Formula \eqref{stirling} is the Stirling number of the second kind, which gives the number of surjective maps from $c(S)$ to $U$.
\end{proof}
Note that $\Triv(S,T)$ contains all the constant maps from $S$ to $T$, but it is not itself a trivial subquandle of $\mathrm{Hom}(S,T)$.

\begin{corollary}
Let $R$ and $S$ be quandles such that $|c(R)|=|c(S)|$ and $T$ be a medial quandle. Then
$\Triv(R,T) \cong \Triv(S,T).$
\end{corollary}

\subsection{From Quandle to Group Homomorphisms} \label{mainsection}

In this section, we extend the isomorphism theorem (Theorem 4.2) in \cite{JPSZ} to provide considerable detailed information about homomorphisms between medial quandles, using the notion from that paper of  ``indecomposable affine mesh,'' which we call an ia-mesh for brevity.
Given a collection of abelian groups $A_i$ for $i$ in some index set $I$, with homomorphisms $\phi_{i,j}:A_i\rightarrow A_j$ and selected elements $c_{i,j}\in A_j,$ the triple $(A_i, \phi_{i,j},c_{i,j})_{i,j\in I}$ is called an {\em ia-mesh}, if the following conditions hold (for arbitrary indices $i,j,j'$, and $k$).
\begin{itemize}
 \item $1-\phi_{i,i}$ is an automorphism of $A_i$ 
  \item $c_{i,i} = 0$
 \item $\phi_{j,k}\circ\phi_{i,j} = \phi_{j',k}\circ\phi_{i,j'}$
 \item $\phi_{j,k}(c_{i,j}) = \phi_{k,k}(c_{i,k}-c_{j,k})$
 \item the elements $c_{i,j}$ and $\phi_{i,j}(a)$ for $i \in I$ and $a \in A_i$ generate the group $A_j$
\end{itemize}
Given an ia-mesh, we can define a binary operation $\rhd$ on the disjoint union of the $A_i$ as follows: for $a\in A_i$ and $b \in A_j,$ $a\rhd b = c_{i,j} + \phi_{i,j}(a) + (1-\phi_{j,j})(b).$ Then Lemmas 3.8 through 3.13 of \cite{JPSZ} may be summarized as:
\begin{theorem} $(\bigcup A_i,\rhd)$ is always a medial quandle with components $\{A_i\}_{i\in I}$, and every medial quandle arises in this fashion.
\end{theorem}

Now we extend Theorem 4.1 from \cite{JPSZ} to understand homomorphisms between medial quandles. Lemma \ref{component to component} tells us that any homomorphism $h:S\to T$ of quandles induces a mapping on their components,
$\hat{h}: c(S) \to c(T),$ defined by $c_S(a)  \mapsto c_T(h(a))$.  We consider the conditions under which one can go in the opposite direction, i.e., lift a mapping $g:c(S)\to c(T)$ to a homomorphism $\tilde{g}: S \rightarrow T$.  

\begin{theorem} \label{maintheorem}
Let $S=\{S_i, \sigma_{i,j},s_{i,j}\}_{i,j\in I}$ and $T=\{T_i, \tau_{i,j}, t_{i,j}\}_{i,j\in J}$ be medial quandles and let $g:I\to J$ be a mapping. Then there exists a homomorphism $h:S\to T$ with $\hat{h}=g$ if and only if for every $i\in I$ there exist group homomorphisms $k_i:S_i\to T_{g(i)}$ and elements $e_i\in T_{g(i)}$ such that for every $i,j\in c(S)$,
\begin{itemize}
\item[(i)] $k_j \circ \sigma_{i,j}=\tau_{g(i),g(j)} \circ k_i$
\item[(ii)]  $k_j(s_{i,j})=t_{g(i),g(j)}+\tau_{g(i),g(j)}(e_i)-\tau_{g(j),g(j)}(e_j)$
\end{itemize}
\end{theorem}

\begin{proof}
$(\Rightarrow)$ Define $e_i$ to be $h(0_i)$, where $0_i$ is the zero element of $A_i$. Then define the maps $k_i$ by $k_i(a) = h(a)-e_i$. We must show that the $k_i$ are group homomorphisms and that the two conditions above are satisfied.  We first note that 
$$h(\sigma_{i,i}(a)) = h(a \rhd 0_i) = h(a) \rhd h(0_i) = \tau_{g(i), g(i)}(h(a)) + (1 - \tau_{g(i), g(i)})(e_i)$$ for $a \in A_i$.  Similarly, $$h((1 - \sigma_{j,j}(b)) = h(0_i \rhd b) = h(0_i) \rhd h(b) = \tau_{g(i), g(i)}(e_i) + (1 - \tau_{g(i), g(i)})(h(b))$$ for $b \in A_i$. Now we show that $k_i$ is a homomorphism.  Let $c,d \in S_i$ and choose $a, b \in A_i$ so that $c = \sigma_{i,i}(a)$ and $d = (1 - \sigma_{i,i})(b)$, which we can do because of the first condition of being an ia-mesh. Then,
\begin{eqnarray*}
k_i(c + d) = h(c + d) - e_i &=& h(a \rhd b) - e_i \\
& = & h(a) \rhd h(b) \\
& = & \tau_{g(i), g(i)} (h(a)) + (1 - \tau_{g(i), g(i)})(h(b)) - e_i \\
& = & h(\sigma_{i, i}(a)) - (1 - \tau_{g(i), g(i)})(e_i) + h((1 - \sigma_{i, i})(b)) - \tau_{g(i), g(i)})(e_i) - e_i \\
& = & h(c) - e_i + h(d) - e_i \\
& = & k_i (c) + k_i (d) 
\end{eqnarray*} 
Now we must check that identities $(i)$ and $(ii)$ hold.  Since $0_i \rhd 0_j = s_{i,j}$ then 
$$k_{j}(s_{i,j}) = k_j(0_i \rhd 0_j) = h(0_i) \rhd h(0_j) - e_j = e_i \rhd e_j - e_j = t_{g(i), g(j)} + \tau_{g(i), g(j)}(e_i) + (1 - \tau_{g(j), g(j)})(e_j) - e_j$$ which establishes $(ii)$.  For $a \in A_i$, $a \rhd 0_j = s_{i, j} + \sigma_{i,j}(a)$.  Then,
\begin{eqnarray*}
k_j(\sigma_{i,j}(a)) &=& k_j(s_{i,j} + \sigma_{i, j}(a)) - k_j(s_{i,j}) \\
&= & h(a) \rhd e_j - e_j -  t_{g(i), g(j)} - \tau_{g(i), g(j)}(e_i) + \tau_{g(j), g(j)}(e_j) \\
& = & t_{g(i), g(j)} + \tau_{g(i), g(j)}(h(a)) + (1 - \tau_{g(j), g(j)})(e_j) - e_j -  t_{g(i), g(j)} - \tau_{g(i), g(j)}(e_i) + \tau_{g(j), g(j)}(e_j) \\
& = &  \tau_{g(i), g(j)}(h(a)) - \tau_{g(i), g(j)}(e_i) \\
& = & \tau_{g(i), g(j)}(h(a) - e_i) = \tau_{g(i), g(j)}(k_i(a))
\end{eqnarray*}
which is $(i)$. \\

$(\Leftarrow)$ Let $a \in S$ and $i = c_S(a)$.  We define $h(a) = k_i(a) + e_i.$  We must show that $h$ is a homomorphism.  For $a \in S_i$ and $b \in S_j$ we have
\begin{eqnarray*}
h(a \rhd b) & = & h (s_{i, j} + \tau_{i,j}(a) + (1 - \tau_{j, j}(b)) \\
& = & k_j(s_{i, j} + \tau_{i, j}(a) + (1 - \tau_{j, j}) (b)) + e_j \\
& = & k_j(s_{i, j}) + k_j(\sigma_{i, j}(a)) + k_j((1 - \sigma_{j, j})(b)) + e_j \\
& = & t_{g(i), g(j)} + \textcolor{red}{\tau_{g(i), g(j)}(e_i)} - \textcolor{blue}{\tau_{g(j), g(j)}(e_j)}  + \textcolor{red}{\tau_{g(i), g(j)}(k_i(a))} + \textcolor{green}{k_j(b)} - \textcolor{blue}{\tau_{g(j), g(j)}(k_j(b))} + \textcolor{green}{e_j} \\
& = & t_{g(i), g(j)} + \textcolor{red}{\tau_{g(i), g(j)}(h(a))} - \textcolor{blue}{\tau_{g(j), g(j)}(h(b))} + \textcolor{green}{h(b)} \\
& = & t_{g(i), g(j)} + \tau_{g(i), g(j)}(h(a)) + (1 - \tau_{g(j), g(j)})(h(b)) \\
& = & h(a) \rhd h(b)
\end{eqnarray*}

\end{proof}

Note that this proof actually shows that we have a bijection between the elements of $\Hom(S,T)$ and sequences of group homomorphisms and elements $(k_i, e_i)$ satisfying properties $(i)$ and $(ii)$.  In this bijection, on an individual component $S_i$, we have $h(a) = k_i(a) + e_i$.  In other words, all quandle homomorphisms between $S$ and $T$ are componentwise affine maps between the components of $S$ and $T$, precisely the ones satisfying conditions (i) and (ii) above.

\subsection{The Source of a Hom Quandle} \label{homsource}
In the previous section we focused on Hom$(S,T)$ when both $S$ and $T$ are medial.  However, by Lemma \ref{factoring} we know that for any quandle $S$, Hom$(S,T)$ is always identical to Hom$(R,T)$ for some medial $R$.  In this section we provide the details to make this statement precise.

We begin with the following definition:
\begin{definition}
Let $Q$ be a quandle and $$C=\setof{((a\rhd b)\rhd(c\rhd d),(a \rhd c)\rhd(b\rhd d))}{a,b,c,d\in Q}$$ We will denote the congruence generated by $C$ by $m_Q$.
\end{definition}
Note that $m_Q$ is the smallest congruence such that the quotient is medial. Moreover, since both terms in the definition of $C$ above belong to $c_Q(d)$, we have $m_Q\subseteq \mathrm{ker}(c_Q)$, and hence $|c(Q)|=|c(Q/m_Q)|$ by Proposition \ref{number_of_comp}.

\begin{theorem} \label{medialquotient}
Let $S$ be a quandle and $T$ a medial quandle. Then $S/m_S$ is medial and $\mathrm{Hom}(S,T)\cong \Hom(S/m_S,T)$ as quandles.
\end{theorem}
\begin{proof}
The proof is analogous to that of Theorem \ref{factoring}, with the observation that the bijection $\psi$ is actually an isomorphism of quandles.
\end{proof}

This result explains why rows $2I$ and $Q4_5$ of Table \ref{table} are identical; $2I$ is the medial quotient of the non-medial $Q4_5$.

Finite connected medial quandles are well studied as they can be understood as finite modules over the Laurent polynomials \cite{Hou, thesis, HSV}.

\begin{lemma}\label{on size}
Let $T$ be a finite connected medial quandle and $S$ be a subquandle of $T$. Then $|S|$ divides $|T|$.
\end{lemma}
\begin{proof} 

Every finite connected medial quandle is Latin and it is isomorphic to a quandle $\mathcal{Q}(A,f)$, where $A$ is an abelian group, $f\in \Aut{A}$ and $a\rhd b= (1-f)(a)+f(b)$ for every $a,b\in T$ \cite{Hou, HSV}. Let $S$ be a subquandle containing $0$. Then for any $b\in S,$ there must also be a $y\in S$ such that $0\rhd y = b,$ namely $y=f^{-1}(b).$ Further, since $S$ is Latin by Lemma \ref{on latin}, for any $a\in S,$ there must be an $x \in S$ such that $x\rhd 0 = a,$ i.e., $(1-f)(x)=a$. Then $S$ contains $x\rhd y = (1-f)(x)+f(y) = a+b.$ Therefore $S$ is a subgroup of $A$. 

If $Q$ is an arbitrary subquandle, choose any $a\in Q$. Then $Q=a+S$ where $S$ is a subquandle containing $0$, since the mappings $b\mapsto a+b$ are automorphisms of $T$. Hence, every subquandle is a subgroup or a coset of a subgroup of $A,$ and so has order dividing that of $T$.
\end{proof}

\begin{theorem}
Let $S$ and $T$ be finite quandles of relatively prime orders with $T$ medial and connected. Then $\mathrm{Hom}(S,T)\cong T$.
\end{theorem}
\begin{proof}
By Lemma \ref{on latin} and Lemma \ref{on size} every subquandle of $T$ is Latin and its size divides the size of $T$. In particular this is true for the image of any homomorphism into $T$.

Let $S$ be a quandle and $h:S\to T$. Suppose $h(a) \rhd h(b) = h(c)$.  Then, for any $b' \in [b]_{\ker(h)}$, $h(a \rhd b') = h(a) \rhd h(b') = h(c)$, so $a \rhd b' \in [c]_{\ker(h)}$.  In other words, $L_a: [b] \rightarrow [c]$.  Since $L_a$ is injective, $|[b]| \leq |[c]|$.  On the other hand, $\Img(h)$ is a connected subquandle of $T$, so this relationship must hold for any pair of equivalence classes. Hence, we conclude 
that all the equivalence classes of $\ker(h)$ have the same size, so $|S|=|\Img(h)||[a]_{ker(h)}|$. Thus, if $|S|$ and $|T|$ are relatively prime,  $|\Img(h)|=1$.
\end{proof}

This corollary explains the preponderance of 3's in the column for $Q3_2$ and 4's in the column for $Q4_7$ in Table \ref{table} as both of these are medial connected quandles.  

The results of this section combined with Theorem \ref{maintheorem} mean that, in some sense, we have determined all of the quandles of the form $\Hom(S,T).$  Given an arbitrary $S$, we form its medial quotient $S/m_S$, realize it as an ia-mesh, and now determine the collections of group homomorphisms $k_i$ and constants $e_i$ as in Theorem \ref{maintheorem}.

\subsection{Homomorphism into 2-reductive quandles} \label{main}
The general observation at the end of the previous Section \ref{homsource} can be made quite concrete in the case of ``2-reductive'' quandles.  A quandle $Q$ is {\it 2-reductive} if $(x \rhd y) \rhd z = y \rhd z$ for all $x,y,z \in Q$.  Jedlicka {\it et al.} in \cite{JPSZ} provide a structure theorem for 2-reductive quandles which we take advantage of in this section to study the space of quandle homomorphisms into such targets.  (Note that \cite{JPSZ} uses a slightly different identity to define 2-reductive; the two definitions are equivalent in the presence of mediality, and the definition above implies mediality. In short, 2-reductive as defined here coincides with ``2-reductive medial" as used in \cite{JPSZ}.)

In particular, Theorems 3.14 and 6.9 of \cite{JPSZ} can be summarized as follows:

\begin{theorem} \label{DavidThm}
\cite{JPSZ} An ia-mesh $(A_i, \phi_{i,j}, c_i)_{i,j\in I}$ gives rise to a 2-reductive quandle if $\phi_{i,j}=0$ for all $i,j\in I$. Moreover, all 2-reductive quandles arise in this way.
\end{theorem}

In light of this result, we will often drop the $\phi_{i,j}$ altogether from an ia-mesh when it is clear that we are concerned with a 2-reductive quandle.
A corollary of Theorem \ref{DavidThm} is that all components of 2-reductive quandles are trivial.  Therefore the only connected 2-reductive quandle is the one-element quandle $I$.  

Again, we can understand arbitrary source quandles by employing a procedure similar to that of Section \ref{homsource} tailored to the 2-reductive case.

\begin{definition}
Let $Q$ be a quandle and 
$$C=\setof{((a\rhd b)\rhd c, b\rhd c)}{a,b,c\in Q}$$ 
We will denote the congruence generated by $C$ by $\gamma_Q$.
\end{definition}
Note that $\gamma_Q$ is the smallest congruence such that the quotient is 2-reductive, and again $\gamma_Q \subseteq \mathrm{ker}(c_Q)$, hence $|c(Q)|=|c(Q/\gamma_Q)|$. So, we have the following:

\begin{theorem}\label{factoring_red}
Let $S$ be a quandle and $T$ a 2-reductive quandle. Then $S/\gamma_S$ is 2-reductive and $\mathrm{Hom}(S,T)\cong \mathrm{Hom}(S/\gamma_S,T)$.  
\end{theorem}

\begin{corollary}
Let $S$ be a connected quandle and $T$ be a 2-reductive quandle. Then every homomorphism between $S$ and $T$ is a constant mapping and $\mathrm{Hom}(S,T) \cong T$.
\end{corollary}
\begin{proof}
We have that $\gamma_S$ is the full relation since $S/\gamma_S$ is a connected 2-reductive quandle. Hence it is trivial and connected, so $|S/\gamma_S|=1$. Thus, $\mathrm{Hom}(S,T)=\mathrm{Hom}(I,T) \cong T$.
\end{proof}

This corollary tells us that 2-reductive quandles will not provide additional information about a knot; they can however be useful for coloring links.  See for example \cite{QTri}.

\begin{lemma}\label{image of base points uniquely determines}
Let $S$ be a quandle, $T$ a 2-reductive quandle, and $h:S \to T$ a homomorphism. Then $h$ is completely determined by the image of a set of base points.  \end{lemma}

\begin{proof}
Let $\{b_i : i \in c(S)\}$ be a set of base points of $S$. Let $h(b_i)=e_i \in \hat{h}(i)$. The images of the base points determine the mapping $\hat{h}$, by virtue of Lemma \ref{component to component}. Then the image of any $a$ in component $i \in c(S)$ is given by:

\begin{eqnarray*}
h(a) & =& h(a_{i_1}\rhd(a_{i_2}\rhd(\ldots (a_{i_n}\rhd b_i))))\\
&=& h(a_{i_1})\rhd(h(a_{i_2})\rhd(\ldots (h(a_{i_n})\rhd h(b_i))))\\
&=& e_i+\sum_{\ell=1}^n t_{\hat{h}(i_\ell),\hat{h}(i)}
\end{eqnarray*}
Thus, $h$ is completely determined by the image of the base points.
\end{proof}

In this setting, Theorem \ref{maintheorem} has a powerful corollary for the case of 2-reductive quandles.

\begin{corollary}\label{criterion2}
Let $S$ be a quandle with base points $\setof{b_i}{i \in c(S)}$, $S/\gamma_S=\{S_i,s_{i,j}\}_{i,j\in c(S)}$ 
and $T = \{T_i, t_{i,j} \}_{i,j\in c(T)}$ be a 2-reductive quandle. Let $g : c(S)\to c(T)$ and $\setof{e_i \in g(i)}{i \in c(S)}\subseteq T$. Then the following are equivalent: 
\begin{itemize}
\item[(i)] there exists a quandle homomorphism $h:S\to T$ such that $\hat{h}=g$ and $h(b_i)=e_i$,
\item[(ii)] for each $i \in c(S)$, the map $k_i:S_i\to T_{g(i)}$ induced by $s_{j,i}\mapsto t_{g(j),g(i)}$
is a group homomorphism,
\item[(iii)] for any selection of elements $f_i \in g(i)$, there exists a unique homomorphism $h:S\to T$ such that $h(b_i) = f_i$.

\end{itemize}
\end{corollary}
\begin{proof}

(i) $\Rightarrow$ (ii) Follows from Theorem \ref{maintheorem} since condition (i) of that result is vacuous in the 2-reductive case and condition (ii) is precisely the definition of the maps $k_i$ above, since the $s_{j,i}$ generate the $S_i$.

(ii) $\Rightarrow$ (iii) As in the previous implication, the $k_i$ give us exactly the collection of group homomorphisms required by Theorem \ref{maintheorem}.  Now, note that with all the $\sigma_{i,j}$ and $\tau_{i,j}$ equal to zero, the constants $e_i$ play no role in the condition.  Hence, the quandle homomorphism $h: S \rightarrow T$ must exist for any choice of $e_i$.  Uniqueness follows from Lemma \ref{image of base points uniquely determines}.

(iii) $\Rightarrow$ (i) Clear.
\end{proof}

This corollary exposes an interesting dichotomy; for a given map of components, there are either no homomorphisms that correspond to that map, or we can map basepoints arbitrarily.  

As a consequence of Corollary \ref{criterion2} we can compute the size of Hom quandles with 2-reductive targets.

\begin{corollary}\label{arbitrary image of base points}
Let $S$ be a quandle and $T$ be a 2-reductive quandle. Then:

$$|\Hom(S, T)|=\sum_{g:c(S)\to c(T)} \delta_g \prod_{i\in c(S)} |g(i)| \quad \textrm{where} \quad \delta_g=\left\{
    \begin{array}{ll}
    	1, \quad \text{ if Thm. \ref{criterion2} (ii) holds for g}\\
        0, \quad \text{otherwise.}
   \end{array}
    \right. $$

\end{corollary}

Note that if all components of $T$ have the same size $n$, then the size of the Hom set from $S$ into $T$ will always be a multiple of $n^{|c(S)|}$.  The $Q4_6$ and $Q6_{67}$ columns in Table \ref{table} illustrate this phenomenon.  The quandle $Q4_6$ has two components of order two, while $Q6_{67}$ has two components of order three, each isomorphic to $3I$. 
For an example of when the components of $T$ are not the same size, consider $Q5_{16}$, which has components of sizes two and three.  Consider the source $Q4_6$.  There are four possible maps $g$, but if the two components of $Q4_6$ are mapped to different components of $Q6_{67}$, then one of the $k_i$ in Corollary \ref{criterion2} (ii) would have to be a homomorphism from $\mathbb{Z}_2$ to $\mathbb{Z}_3$.  Hence, only the $g$ which map both components of the source to a single component in the target contribute to the sum, and we have $2 \times 2 \,+\, 3 \times 3 = 13$ homomorphisms as indicated in Table \ref{table}.

The conditions on $T$ in Theorem \ref{criterion2} may seem quite specialized.  However, Tables 1 and 2 of \cite{JPSZ} show that the vast bulk of quandles are, in fact, 2-reductive, perhaps asymptotically all of them.

\begin{corollary}\label{structure of hom}
Let $S$ be a quandle and $T$ be a 2-reductive medial quandle. Then $\Hom(S,T)$ embeds in $T^{c(S)}$ and $\Hom(S,T)= \bigcup_{i} X_{i}$ where each $X_{i}$ is a component of $T^{c(S)}$.  
\end{corollary}

\begin{proof}
Let $\{b_i : i \in c(S)\}$ be a set of base points of $S$. The embedding $k: \Hom(S,T) \rightarrow T^{c(S)}$ defined by $h \mapsto (h(b_1),\ldots,h(b_n))$ is injective by Corollary \ref{criterion2} (iii) and is a homomorphism due to the component-wise definition of the quandle structure on $\Hom(S,T)$.  The components of $T^{c(S)}$ are given by $\displaystyle{C = \prod_{i\in c(S)} T_{j_i}}$ where $T_{j_i}$ are components of $T$. Let $k(h) \in C$, i.e., $h(b_i) =e_i \in T_{j_i}$ for every $i \in c(S)$.   Then by (iii) of Corollary \ref{criterion2} with $g(i) = T_{j_i}$, we have $C \subset \Img(k)$.
\end{proof}

Note that the embedding given in Corollary \ref{structure of hom} refines the embedding in Theorem 8 of \cite{CN}, since the minimum cardinality of a set of generators is at least $|c(S)|$ (since we need at least one element from each component to generate the quandle). Note that if $S$ is already 2-reductive, then it is generated by any set of basepoints, so the minimum number of generators is equal to the number of components.  Moreover, the isomorphic copy of $\Hom(S,T)$ in $T^{c(S)}$ is a union of components, not merely an arbitrary subquandle.

\begin{corollary}\label{TRIV}

Let $S$ be a quandle and $T$ be a 2-reductive quandle. Let $$G =\{ g: c(S) \rightarrow c(T) : t_{g(i), g(j)} = 0 \, \textrm{ for all } i, j \in c(S)\}$$  Then $$\Triv(S,T) \cong \bigcup_{g \in G}  \prod_{i \in c(S)} T_{g(i)}$$ In particular $\bigcup_{i \in c(T)} T_i^{c(S)} \leq \Triv(S,T)$.
\end{corollary}

\begin{proof}
Let $h\in \Triv(S,T)$. Then $h(e_j) = h(e_i)\rhd h(e_j)=h(e_j)+t_{\hat{h}(i),\hat{h}(j)}$, so $t_{\hat{h}(i), \hat{h}(j)}=0$ for every $i,j\in c(Q)$. Thus, $\hat{h} \in G$. On the other hand, if $g \in G$ then each $k_i$ of Theorem \ref{criterion2} (ii) is the constant zero map, which is always a homomorphism.  Hence, $\prod_{i \in c(S)} T_{g(i)} \leq \Triv(S,T).$

\end{proof}

If all the components of the source are connected quandles, the homomorphisms are determined by trivial subquandles of the target.

\begin{corollary}\label{components are connected}
Let $S$ be a quandle and $T$ be a 2-reductive quandle. If all the components of $S$ are themselves connected, then 
$\Hom(S,T)=\Triv(S,T)$.
\end{corollary}
\begin{proof}
Since the components of $S/\gamma_S$ are connected and trivial, then $S/\ker(c(S))=S/\gamma_S$. So $S/\gamma_S$ is trivial and therefore, for any $h \in \Hom(S,T)$, $\Img(h)$ is a trivial subquandle of $S$. 

\end{proof}
We include $Q6_{52}$ in Table \ref{table} to illustrate this corollary.  It has two components, each of which is isomorphic to the connected $Q3_2$.  The homomorphisms from $Q6_{52}$ into a 2-reductive quandle, such as $Q4_6$, are therefore simply the trivial ones.  We can count them using Lemma \ref{hom with projec image}:  one each for the four single-element subquandles of $Q4_6$, and two each for the two trivial two-element subquandles.

\begin{example}
This example illustrates how $\Hom(S,T)$ depends in a significant way on the constants in the ia-mesh for $S$. Let $T=\left\{ \mathbb{Z}_n,\mathbb{Z}_m, t =\begin{bmatrix}
0 & 1_m\\
1_n & 0\\
\end{bmatrix}\right\}$ and let $S=\{S_i, s_{i,j}\}_{i,j\in I}$ be a medial 2-reductive quandle. By Corollary \ref{TRIV}, we have $\Triv(S,T)=\mathbb{Z}_n^{c(S)}\cup \mathbb{Z}_m^{c(S)}$. Let $h\in \Hom(S,T)$.  If $s_{i,j}=0$, then $k_j (s_{i,j})=t_{\hat{h}(i),\hat{h}(j)}=0$, and therefore $\hat{h}(i)=\hat{h}(j)$. In particular if any component $S_k$ of $S$ acts trivially, i.e., $s_{k,j}=0$ for every $j\in S$, then $\hat{h}(i)=\hat{h}(k)$ for every $i\in I$, which means $\Hom(S,T)=\Triv(S,T)$. 
\end{example}

Taken together, these results explain all of the entries in Table \ref{table} except for the column of $Q7_{71}$, the only medial quandle in the table that is not 2-reductive.  However, since none of the quandles in Table \ref{table} have more than two non-singleton components, it is possible to see using Theorem \ref{maintheorem} that for all of the sources into $Q7_{71}$ the images of the homomorphisms lie in a single component of $Q7_{71}$.  Hence, in fact, we can count them using Corollary \ref{criterion2}.  Presumably this will not be the case for larger sources and/or targets that are not 2-reductive.  

Corollary \ref{structure of hom} gives us the precise structure of $\Hom(S,T)$ when $T$ is 2-reductive as a union of products of components of $T$.  An interesting direction for further investigation would be to understand the behavior of the $(k_i, e_i)$ of Theorem \ref{maintheorem} sufficiently well to provide a similarly precise characterization of  $\Hom(S,T)$ when $T$ is not necessarily 2-reductive.  

\section{Acknowledgements}
Some of the initial ideas that led to this paper arose in a conversation between the authors and David Stanovsk\'{y}. We thank him for this insightful and helpful discussion, particularly the 
guidance to focus attention on 2-reductive quandles.


\begin{thebibliography}{AAAA}

\bibitem{AG} Andruskiewitsch, N. and Grana, M. ``From racks to pointed Hopf algebras." {\it Advances in Mathematics,} Vol. 178 (2003), 

\bibitem{CN} Crans, A. and Nelson, S. ``Hom Quandles." {\it Journal of Knot Theory and its Ramifications,} Vol. 23 (2014) No. 2



\bibitem{QTri} Elhamdadi, M., Liu, M.,  and Nelson, S. ``Quasi-trivial Quandles and Biquandles, Cocycle Enhancements and Link-Homotopy of Pretzel links." {\verb+arXiv:1704.01224+}

\bibitem{Elhamdadi} Elhamdadi, M., Macquarrie, J., and Restrepo, R. ``Automorphism Groups of Quandles." {\it Journal of Algebra and its Applications,} Vol. 11 (1) (2010)


\bibitem{EGS} Etingof P., Guralnik R., and Soloviev A. ``Indecomposable set-theoretical solutions to the Quantum Yang-Baxter Equation on a set with prime number of elements." {\it Journal of Algebra,}  242 (2001)


\bibitem{G} Gra\~na M. ``Indecomposable racks of order $p^2$." {\it Beitr\"age zur Algebra und Geometrie,} Contributions to Algebra and Geometry 45, No. 2 (2004)

\bibitem{thesis} Holmes, H.  ``Left distributive algebras and knots." Master's Thesis, Charles University in Prague (2013). Available
at {\verb+https://is.cuni.cz/webapps/zzp+}

\bibitem{HN} Ho, B. and Nelson, S. ``Matrices and Finite Quandles." {\it Homology, Homotopy and Applications,} Vol. 7 (1) (2005) 

\bibitem{Hou} Hou X.D. ``Finite Modules over $\mathbb{Z}[t,t^{-1}]$." {\it Journal of Knot Theory and Its Ramifications,} Vol. 21, No. 8.



\bibitem{HSV}  Hulpke A., Stanovsk\`{y} D., Vojt\v{e}chovsk\'{y} P. ``Connected quandles and transitive groups." {\it Journal of Pure and Applied Algebra,} 220, No. 2 (2016)  

\bibitem{JPSZ} Jedlicka, P., Piltowska, A., Stanovsk\`{y} D., and Zamojska-Dzienio, A. ``The Structure of Medial Quandles." {\it Journal of Algebra,} Vol. 443 (2015), pp. 300 -- 334.

\bibitem{SI} Jedlicka, P., Piltowska, A., Stanovsk\`{y}, D., and Zamojska-Dzienio, A.  ``Subdirectly irreducible medial quandles." {\verb+arXiv:1511.06529v3+}

\bibitem{Jsimple} Joyce, D. ``Simple quandles." {\it Journal of Algebra,} 79, No. 2 (1982)

\bibitem{Lith} Litherland R.A., and Nelson S. ``The Betti numbers of some finite racks." {\it Journal of Pure and Applied Algebra,} 178 Issue 2, (2003) 

\bibitem{Mur} Murdoch, D.C. ``Structure of abelian quasigroups." {\it Transactions of the American Mathematical Society,} 49 (3) (1941) 

\bibitem{RS} Romanowska A.B., Smith J.D.H., \emph{Modes}, World Scientific, Singapore (2002).

\bibitem{Sta-latin} Stanovsk\'{y} D. ``A guide to self-distributive quasigroups or Latin quandles."  {\it Quasigroups and Related Systems,} 23 (2015)

\bibitem{V} Vendramin, L. ``Rig, a GAP package for racks, quandles and Nichols algebras."  Available at {\verb+http://github.com/vendramin/rig/+}

\end{thebibliography}
\end{document}